\documentclass[a4paper,11pt]{amsart}

\usepackage[T1]{fontenc}

\usepackage{microtype}
\usepackage{lmodern}

\usepackage[foot]{amsaddr}

\usepackage{amssymb}
\usepackage{amsmath}
\usepackage{amsthm}
\usepackage[USenglish]{babel}

\usepackage{constants}
\usepackage{paralist}

\usepackage{esint}
\usepackage[final]{hyperref}

\usepackage[abbrev,backrefs]{amsrefs}

\usepackage[nameinlink%
,capitalize%
]{cleveref}

\usepackage[latin1]{inputenc}

\usepackage[abbrev]{amsrefs}

\usepackage{geometry}

\allowdisplaybreaks

\newtheorem{theorem}{Theorem}[section]
\newtheorem{lemma}[theorem]{Lemma}
\newtheorem{proposition}[theorem]{Proposition}

\theoremstyle{plain}
\newtheorem{claim}{Claim}
\newcommand{\resetclaim}{\setcounter{claim}{0}}
\newenvironment{proofclaim}
  [1]
  [Proof of the claim]
  {\begin{proof}[#1]}
  {\end{proof}}

\theoremstyle{definition}
\newtheorem{definition}[theorem]{Definition}

\theoremstyle{remark}
\newtheorem{remark}[theorem]{Remark}

\numberwithin{equation}{section}

\newcommand{\abs}[1]{\lvert #1 \rvert}
\newcommand{\bigabs}[1]{\bigl\lvert #1 \bigr\rvert}

\newcommand{\biggabs}[1]{\biggl\lvert #1 \biggr\rvert}

\newcommand{\norm}[1]{\lVert #1 \rVert}

\newcommand{\st}{\;\vert\;}
\newcommand{\dif}{\,\mathrm{d}}
\newcommand{\N}{\mathbb N}
\newcommand{\R}{\mathbb R}

\allowdisplaybreaks

\begin{document}

\title[Choquard equation with sign-changing potential]{
Groundstates of the Choquard equations with a sign-changing self-interaction potential}

\author{Luca Battaglia}
\address{Universit\`a degli Studi Roma Tre\\
Dipartimento di Matematica e Fisica\\
Largo S. Leonardo Murialdo 1\\
00146 Rome\\
Italy}
\email{lbattaglia@mat.uniroma3.it}

\author{Jean Van Schaftingen}
\address{Universit\'e catholique de Louvain\\ 
Institut de Recherche en Math\'ematique et Physique\\
Chemin du Cyclotron 2 bte L7.01.01\\
1348 Louvain-la-Neuve\\
Belgium}
\email{Jean.VanSchaftingen@uclouvain.be}

\thanks{This work was supported by the Projet de Recherche (Fonds de la Recherche Scientifique--FNRS) T.1110.14 ``Existence and asymptotic behavior of solutions to systems of semilinear elliptic partial differential equations''.}
\date{}

\begin{abstract}
We consider a nonlinear Choquard equation
$$
  -\Delta u+u= (V * |u|^p )|u|^{p-2}u \qquad \text{in }\mathbb{R}^N,
$$
when the self-interaction potential $V$ is unbounded from below.
Under some assumptions on \(V\) and on \(p\), covering \(p =2\) and \(V\) being the one- or two-dimensional Newton kernel, we prove the existence of a nontrivial groundstate solution $u\in H^1 (\mathbb{R}^N)\setminus\{0\}$ 
by solving a relaxed problem by a constrained minimization and then proving the convergence of the relaxed solutions to a groundstate of the original equation.
\end{abstract}

\maketitle

\section{Introduction}

We are interested in the nonlinear Choquard equation``
\begin{equation}\label{equationChoquard}
-\Delta u+u=\bigl(V * \abs{u}^p \bigr)\,\abs{u}^{p-2}u;\tag{$\mathcal{C}$}
\end{equation}
in the Euclidean space \(\R^N\) with \(N \in \mathbb{N} = \{1, 2, \dotsc\}\), where \(p \in [1, +\infty)\) is a given exponent and \(V : \R^N \to \R\) is a given self-interaction potential. 
Solutions to the Choquard equation \eqref{equationChoquard} correspond, \emph{at least formally}, to critical points of the \emph{functional} \(\mathcal{I}\) defined for each function \(u : \R^N \to \R\) by 
\begin{equation}\label{i}
\mathcal{I} (u)=\frac{1}{2}\int_{\R^2}\bigl(\abs{\nabla u}^2+u^2\bigr)-\frac{1}{2p}\int_{\R^2}\bigl(V * \abs{u}^p \bigr)\,\abs{u}^p.
\end{equation}

When \(p = 2\), \(N = 3\) and the self-interaction potential \(V\) is \emph{Newton's kernel}, that is, the fundamental solution of the Laplacian on the space \(\R^3\), the Choquard equation \eqref{equationChoquard} arises in several fields of physics (quantum mechanics \cite{pek}, one-component plasma \cite{lie}, self gravitating matter \cite{mpt}). In the more general setting where \(N \in \N_*\), that the function \(V : \R^N \to \R\) is a \emph{Riesz potential}, that is, for \(x \in \R^N \setminus \{0\}\),
\[
 V (x) = I_\alpha (x) 
  := \frac{\Gamma \bigl(\frac{N-\alpha}{2}\bigr)}
  {\Gamma \bigl(\frac{\alpha}{2}\bigr)\pi^\frac{N}{2} 2^\alpha |x|^{N-\alpha}},
\]
with $\alpha \in (0,N)$ and that \(\frac{N + \alpha}{N} \le \frac{1}{p} \le \frac{N + \alpha}{N - 2}\), the existence of a groundstate solution, minimizing the functional \(\mathcal{I}\) among all solutions and of multiple solutions has been proved \citelist{\cite{lie}\cite{m80}\cite{mv13}\cite{Stuart1980}\cite{ChoquardStubbeVuffray2008}\cite{TodMoroz1999}\cite{Bongers1980}\cite{lio}\cite{Lions1982}}. Similar results have been obtained when the power nonlinearity  $\abs{u}^p$ is replaced by a more general nonlinearity \citelist{\cite{mv15}\cite{bv}}. 
We also refer the interested reader to the recent survey \cite{mv17}.

In \emph{low dimensions} \(N \in \{1, 2\}\), the Newton kernel is not anymore a Riesz potential and is characterized by a \emph{linear} or \emph{logarithmic growth at infinity}. This means that the above results cannot be transferred readily to the case where \(V\) is the one- or two-dimensional Newton kernel and that other ideas and methods are needed. 
Solutions for the Choquard equation with a low-dimensional Newton kernel have been constructed by ordinary differential equation methods \cite{ChoquardStubbeVuffray2008}.

A key difficulty in order to construct solutions variationally 
is that the functional \(\mathcal{I}\) \emph{is not well-defined on the natural Sobolev space} \(H^1 (\R^N)\).
For the two-dimensional logarithmic Newton kernel,   variational methods
have been applied successfully in the framework of the Hilbert space of functions \(u : \R^N \to \R\) such that 
\begin{equation}
\label{eqCondHilbert}
 \int_{\R^2} \bigl( \abs{\nabla u}^2 + \bigl(1 + V^-\bigr) \abs{u}^2 \bigr)< + \infty
\end{equation}
(see \citelist{\cite{ChoquardStubbe2007}
\cite{Stubbe2008}\cite{StubbeVuffray2010}\cite{dw}\cite{Cingolani_Weth}\cite{bcv}}).
A delicate point in this approach is that although the Choquard equation \eqref{equationChoquard} and the associated energy functional \(\mathcal{I}\) are \emph{invariant under translations} of the Euclidean space \(\R^N\), the Hilbert space naturally defined by the quadratic form \eqref{eqCondHilbert} is \emph{not} any more \emph{invariant} under translation. 
At the technical level, this means, for example, that a translated Palais--Smale sequence need not be itself a Palais--Smale sequence.
This nonintrinsic character of the method translates in a rigidity of the results, that cannot be readily generalized to similar classes of potentials.

The goal of the present work is to develop a method invariant under translations that would yield more flexible existence results for the Choquard problem \eqref{equationChoquard}.

Since the main issue in studying variationally \eqref{equationChoquard} is the unboundedness of the negative part of the pontetial \(V^-\), we propose to construct solutions to a relaxed variational problem obtained by truncating the potential \(V\) and then to obtain solutions of the original problem by passing to the limit on the relaxation parameter.

We obtain the following result:

\begin{theorem}\label{theoremMainGroundstate}
Let \(N \in \N\) and \(p \in [2, +\infty)\).
If \(p \ge 2\) and if the function \(V \in \R^N \to \R\) is even and satisfies
\begin{align}
\tag{$V_1$}
\label{V1} 
& V^+ \in L^q (\R^N) 
  \text{ with }q \in [1, +\infty)
  \text{ and }\tfrac{1}{p} (1 - \tfrac{1}{2q}) \ge \tfrac{1}{2} - \tfrac{1}{N},\\
\tag{$V_2$} 
\label{V2} & \sup_{\abs{x - y} \le 1} \abs{V^- (x) - V^- (y)} < + \infty,\\
\tag{$V_3$} 
\label{V3} 
& \text{there exists a function \(\varphi \in C^1_c (\R^N)\) such that}
\int_{\R^N} \bigl(V \ast \abs{\varphi}^p\bigr) \abs{\varphi}^p > 0,\\
\tag{$V_4$} 
\label{V4} 
& \lim_{\abs{x} \to \infty} V (x) = - \infty,
\end{align}
then the Choquard equation \eqref{equationChoquard} has a groundstate weak solution \(u \in H^1(\R^N)\) such that \(\int_{\R^N} (V \ast \abs{u}^p) \abs{u}^p > -\infty\).
\end{theorem}

Here and in the sequel, we have set 
\(V^+ = \max (V, 0)\) and \(V^- = \max (-V, 0)\) so that 
\(V = V^+ - V_-\).

The notion of weak solution in \cref{theoremMainGroundstate} that we use comes with suitable integrability assumptions in order to ensure that the equation makes sense.

\begin{definition}
\label{definitionWeak}
The function \(u \in H^1_{\mathrm{loc}} (\R^N)\) is \emph{a weak solution to the Choquard equation} \eqref{equationChoquard} whenever \((V \ast \abs{u}^p) \abs{u}^{p - 2} u \in H^{-1}_{\mathrm{loc}} (\R^N)\) and, for every test function \(\varphi \in H^1 (\R^N)\) supported in a compact set, one has
\[
  \int_{\R^N} \nabla u \cdot \nabla \varphi + u \,\varphi
  = \int_{\R^N} (V \ast \abs{u}^p) \abs{u}^{p - 2} u\, \varphi.
\]
\end{definition}

Examples of self-interaction potentials \(V : \R^N \to \R\) satisfying the assumptions of \cref{theoremMainGroundstate} are \(V (x) = \frac{1}{2\pi} \ln \frac{1}{\abs{x}}\), corresponding when \(N = 2\) to the Newtonian kernel, \(V (x) = 1 - \frac{\abs{x}}{2}\) corresponding when \(N = 1\) to the Newtonian kernel and more exotic examples of the form \(V (x) = \abs{x}^\alpha - \abs{x}^\beta\) with $\max\{(N-2)p-2N,-N\}<\alpha < \beta \le 1$. 
\Cref{theoremMainGroundstate} also allows for \emph{anisotropic potentials} such as
\(V (x) = \frac{1}{2 \pi} \log \frac{1}{\abs{A x}}\) with \(A : \R^n \to \R^n\) linear and invertible or \(V (x) = \frac{1}{2 \pi} \log \abs{x} - \frac{\kappa}{2 \pi} \log \abs{A x}\), with \(\kappa \in (0, 1)\); their three-dimensional counterparts have been studied recently \cite{Ricaud2016}.

The evenness of the potential \(V\) is required to ensure the variational character of the problem. The boundedness assumption \eqref{V1} on the function \(V\) ensures that the convolution term takes a well-defined, possibly infinite, value in \([-\infty, +\infty)\). It also implies that \(\frac{1}{p} \ge \frac{1}{2} - \frac{1}{N}\). The assumption \eqref{V2} is a \emph{coarse continuity} assumption which governs the large-scale continuity of the function \(V\) \cite{Roe2003}. This condition plays a role in the estimates on the convolution term.

The condition \eqref{V3} is necessary since any solution \(u\) should satisfy the identity
\[
 \int_{\R^N} \bigl( \abs{\nabla u}^2 + \abs{u}^2 \bigr)
 =\int_{\R^N} \bigl(V \ast \abs{u}^p\bigr) \abs{u}^p,
\]
obtained by testing the Choquard equation \eqref{equationChoquard} against \(u\)  (see \cref{lemmaTest}). One way to ensure the condition \eqref{V3} is to require the potential \(V\) to be positive in a neighborhood of \(0\).

The groundstate solution in \cref{theoremMainGroundstate}, is constructed as an extremal function for the supremum
\begin{equation}\label{a}
a:=\sup\, \Bigl\{\int_{\R^N}\bigl(V * \abs{u}^p \bigr)\,\abs{u}^p
\st u\in H^1\bigl(\R^N\bigr) \text{ and }\int_{\R^N}\bigl(\abs{\nabla u}^2+\abs{u}^2\bigr)=1 \Bigr\}.
\end{equation}
Due to the homogeneity of the nonlinear term, one can easily deduce that a multiple of such an extremal is a solution of \eqref{equationChoquard} at the energy level
\begin{equation}\label{b}
b:=\inf \Bigl\{\sup_{t>0}\mathcal{I} (tu) \st 
u\in H^1\bigl(\R^N\bigr)\setminus\{0\}  \Bigr\};
\end{equation}
It is also easily seen to be a ground state, in the sense we previously expressed.

%
%
%

The content of the paper is the following:
in \cref{sectionGroundStateRelaxed} we study the groundstates of a relaxed problem and in \cref{sectionGroundState} we prove the convergence of such ground states to ground states for the Choquard equation \eqref{equationChoquard}.

\section{Groundstate solutions for the relaxed problem}
\label{sectionGroundStateRelaxed}

In order to construct solutions to the Choquard equation \eqref{equationChoquard}, we introduce the parameter \(\lambda \in [0, + \infty)\) and we define a \emph{relaxed potential} \(V_\lambda : \R^N \to \R\) defined for each \(x \in \R^N\) by 
\begin{equation}
\label{eqDef_relaxed_potential}
 V_\lambda (x) := \max \big\{ V (x), {-\lambda} \bigr\}.
\end{equation}
The \emph{relaxed problem} is obtained from  the original Choquard equation \eqref{equationChoquard} by replacing the original potential \(V\) by the relaxed potential \(V_\lambda\):
\begin{align}
\label{problemRelaxed}%
\tag{$\mathcal{C}_\lambda$}%
-\Delta u+u&=\bigl(V_\lambda *\abs{u}^p\bigr)\abs{u}^{p-2}u &
& \text{in }\R^N.
\end{align}

A fundamental tool in our analysis is Young's convolution inequality (see for example \citelist{\cite{LiebLoss2001}*{Theorem 4.2}\cite{Bogachev2007}*{Theorem 3.9.2}}), which plays the same role as the Hardy--Littlewood--Sobolev inequality for the Choquard equation with a Riesz self-interaction potential.

\begin{proposition}[Young's convolution inequality]\label{young}
If \(V \in L^q (\R^N)\) and \(f \in L^r (\R^N)\) with \(1 < \frac{1}{q} + \frac{1}{r} \le 2\), then \(V \ast f \in L^s (\R^N)\) 
with \(\frac{1}{s} = \frac{1}{q} + \frac{1}{r} - 1\) and 
\[
 \int_{\R^N} \abs{V \ast f}^s
 \le \biggl(\int_{\R^N} \abs{V}^q \biggr)^\frac{s}{q}
 \biggl(\int_{\R^N} \abs{f}^r \biggr)^\frac{s}{r}.
\]
In particular, if \(f \in L^{\frac{2q}{2 q - 1}} (\R^N)\), then 
\[
 \biggabs{\int_{\R^N} (V \ast f)f}
 \le  \biggl(\int_{\R^N} \abs{V}^q \biggr)^\frac{1}{q}
 \biggl(\int_{\R^N} \abs{f}^\frac{2q}{2q - 1} \biggr)^{2 - \frac{1}{q}}.
\]
\end{proposition}



As a first application of the estimate of \cref{young}, we have the well-definiteness of our relaxed variational problem.

\begin{lemma}
\label{lemma_Relaxed_Functional_Well_Defined}
Let \(N \in \N\), \(p \in (1, +\infty)\) and \(V : \R^N \to \R\) be an even measurable function.
If \(p \ge 2\) and if \(V^+ \in L^q (\R^N)\) with \(\frac{1}{p}(1 - \frac{1}{2 q}) \ge \frac{1}{2} - \frac{1}{N}\), then for every \(\lambda \in [0, + \infty)\), the functional \(\mathcal{J}_{\lambda} : H^1 (\R^N) \to \R\) defined for each \(u \in H^1 (\R^N)\) by
\[
 \mathcal{J}_\lambda (v) = \int_{\R^N} \bigl(V_\lambda \ast \abs{u}^p \bigr) \abs{u}^p
\]
is well-defined and continuously differentiable.
\end{lemma}

\begin{proof}%
[Proof of \cref{lemma_Relaxed_Functional_Well_Defined}]
By Young's inequality for convolution  (\cref{young}), we have 
\[
 \biggabs{\int_{\R^N} \bigl(V_\lambda^+ \ast \abs{u}^p \bigr) \abs{u}^p}
 \le \biggl(\int_{\R^N} \abs{V^+}^q \biggr)^\frac{1}{q}
 \biggl(\int_{\R^N} \abs{u}^\frac{2 p q}{2q - 1} \biggr)^{2 - \frac{1}{q}}.
\]
The second factor is controlled by the Sobolev inequality
since by our assumption, we have
\(
  \frac{1}{2} - \frac{1}{N} \le \frac{1}{p}\Bigl(1 - \frac{1}{2 q}\Bigr) \le \frac{1}{p} \le \frac{1}{2}
\).
On the other hand, since \(0 \le V_{\lambda}^- \le \lambda\) on \(\R^N\), we have 
\[
 \biggabs{\int_{\R^N} \bigl(V_\lambda^- \ast \abs{u}^p \bigr) \abs{u}^p}
 \le \lambda \biggl(\int_{\R^N} \abs{u}^p\biggr)^2.
\]
which is is controlled by the Sobolev embedding whenever
\(
 \frac{1}{2} - \frac{1}{N} \le \frac{1}{p} \le \frac{1}{2}
\).

The continuous differentiability follows from the same estimates together with the fact that the superposition mapping \(u \in H^1 (\R^N) \mapsto \abs{u}^{p - 2} u \in L^r (\R^N)\) is continuously differentiable whenever \(\frac{1}{2} - \frac{1}{N} \le \frac{p - 1}{r}\le \frac{1}{2}\).
\end{proof}

We are going to construct a multiple of a solution of the relaxed problem \eqref{problemRelaxed} by showing that the following supremum is achieved:
\begin{equation}%
\label{eqDef_a_lambda}
a_\lambda
  :=\sup \,\biggl\{ \int_{\R^N}\bigl(V_\lambda*\abs{u}^p\bigr)\abs{u}^p 
  \st u\in H^1\bigl(\R^N\bigr) 
  \text{ and } \int_{\R^N}\bigl(\abs{\nabla u}^2+\abs{u}^2\bigr)=1 \biggr\}.
\end{equation}

In order to prove this we rely on a Brezis--Lieb inequality for the relaxed potential \(V_\lambda\).

\begin{lemma}
\label{lemmaBrezisLieb}
Let \(N \in \N\), \(p, q \in [1, +\infty)\) and \(V : \R^N \to \R\) be an even measurable function.
If \(V^+ \in L^q (\R^N)\), if the sequence \((u_n)_{n \in \N}\) converges almost everywhere to \(u : \R^N \to \R\) 
and is bounded in \(L^{p} (\R^N) \cap L^\frac{2 pq}{2 q - 1} (\R^N)\),
then for every \(\lambda \in [0, +\infty)\), we have
\begin{multline*}
  \lim_{n\to+\infty}
  \int_{\R^N}\bigl(V_\lambda*\abs{u_n}^p\bigr)\abs{u_n}^p-\int_{\R^N}\bigl(V_\lambda*\abs{u_n - u}^p\bigr)\abs{u_n - u}^p + 2  \int_{\R^N} (V_\lambda^- \ast \abs{u}^p) \abs{u_n - u}^p\\
  = \int_{\R^N}\bigl(V_\lambda *\abs{u}^p\bigr)\abs{u}^p.
\end{multline*}
\end{lemma}

\Cref{lemmaBrezisLieb} is a nonlocal version of the classical Brezis--Lieb lemma \cite{BrezisLieb1983}.
Similar identities and inequalities have been proved when the self-interaction potential \(V\) has constant sign \citelist{%
\cite{Ackermann2006}*{\S 5.1}%
\cite{mv13}%
\cite{BellazziniFrankVisciglia}%
\cite{YangWei2013}*{lemma 3.2}%
\cite{MercuriMorozVanSchaftingen}%
}.
It implies the following interesting Brezis--Lieb inequality
\[
  \limsup_{n\to+\infty}
  \int_{\R^N}\bigl(V_\lambda*\abs{u_n}^p\bigr)\abs{u_n}^p-\int_{\R^N}\bigl(V_\lambda*\abs{u_n - u}^p\bigr)\abs{u_n - u}^p
  \le\int_{\R^N}\bigl(V_\lambda *\abs{u}^p\bigr)\abs{u}^p.
\]

\begin{proof}[Proof of \cref{lemmaBrezisLieb}]
We treat separately the positive and negative contributions of the self-interaction potential \(V\). For the positive part, we follow essentially the argument for the Riesz potential \cite{mv13}*{Lemma 2.4}. 
Since by assumption the sequence \((u_n)_{n \in \N}\) is bounded in \(L^{2pq/(q - 1)}( \R^N)\) and converges almost everywhere to \(u\), the sequence \((\abs{u_n - u}^p)_{n \in \N}\) converges weakly to \(0\) in \(L^{2q/(2 q - 1)} (\R^N)\) (see for example \citelist{ \cite{Bogachev2007}*{Proposition 4.7.2}\cite{Willem2013}*{
Proposition 5.4.7}}) and by the classical Brezis--Lieb lemma \cite{BrezisLieb1983}*{Theorem 1}, the sequence $(\abs{u_n}^p-\abs{u_n - u}^p)_{n \in \N}\) converges strongly to the function \(\abs{u}^p\) in $L^{2q/(2q-1)}\bigl(\R^N\bigr)$.
By Young's convolution inequality (\cref{young}), the sequence $(V^+*(\abs{u_n}^p-\abs{u_n - u}^p))_{n \in \N}\) converges strongly to \(V^+*\abs{u}^p$ in $L^{2q}\bigl(\R^N\bigr)$. Therefore, since the function \(V^+\) is even,
\begin{multline}
\label{eqBLpos}
  \lim_{n \to \infty} 
    \int_{\R^N}\bigl(V^+*\abs{u_n}^p\bigr)\abs{u_n}^p-\int_{\R^N}\bigl(V^+*\abs{u_n - u}^p\bigr)\abs{u_n - u}^p\\
=
  \lim_{n \to \infty} 
    \int_{\R^N}
      \bigl(V^+*(\abs{u_n}^p-\abs{u_n - u}^p)\bigr)\bigl((\abs{u_n}^p-\abs{u_n - u}^p)+2\abs{u_n - u}^p\bigr)\\
  = 
    \int_{\R^N}\bigl(V^+*\abs{u}^p\bigr)\abs{u}^p.
\end{multline}

For the negative part, we first rewrite for each \(n \in \N\) the integrals as 
\begin{multline*}
\int_{\R^N}\bigl(V_\lambda^-*\abs{u_n}^p\bigr)\abs{u_n}^p-\int_{\R^N}\bigl(V_\lambda^-*\abs{u_n - u}^p\bigr) \abs{u_n - u}^p - 2  \int_{\R^N} (V_\lambda^- \ast \abs{u}^p) \abs{u_n - u}^p\\
= \int_{\R^N}\bigl(V_\lambda^-*(\abs{u_n}^p - \abs{u_n - u}^p)\bigr)(\abs{u_n}^p - \abs{u_n - u}^p)\\
+ 2 \int_{\R^N} (V_\lambda^- \ast (\abs{u_n}^p - \abs{u_n - u}^p - \abs{u}^p)) \abs{u_n - u}^p.
\end{multline*}
Since the sequence \((u_n)_{n \in \N}\) is bounded in \(L^p (\R^N)\), by the classical Brezis--Lieb lemma again, the sequence \((\abs{u_n}^p - \abs{u_n - u}^p)_{n \in \N}\) converges strongly to \(\abs{u}^p\) in \(L^1 (\R^N)\), and thus, 
\[
 \lim_{n \to \infty} \int_{\R^N}\bigl(V_\lambda^-*(\abs{u_n}^p - \abs{u_n - u}^p)\bigr)(\abs{u_n}^p - \abs{u_n - u}^p) = \int_{\R^N}\bigl(V_\lambda^-*\abs{u}^p\bigr)\abs{u}^p
\]
and
\[
 \lim_{n \to \infty} \int_{\R^N} (V_\lambda^- \ast (\abs{u_n}^p - \abs{u_n - u}^p - \abs{u}^p)) \abs{u_n - u}^p = 0.
\]
Hence, we obtain
\begin{multline}
\label{eqBLneg}
\lim_{n \to \infty} \int_{\R^N}\bigl(V_\lambda^-*\abs{u_n}^p\bigr)\abs{u_n}^p-\int_{\R^N}\bigl(V_\lambda^-*\abs{u_n - u}^p\bigr) \abs{u_n - u}^p - 2  \int_{\R^N} (V_\lambda^- \ast \abs{u}^p) \abs{u_n - u}^p\\
=  \int_{\R^N}\bigl(V_\lambda^-*\abs{u}^p\bigr)\abs{u}^p.
\end{multline}

The conclusion follows from the combination of the identities \eqref{eqBLpos} and \eqref{eqBLneg}.
\end{proof}

\begin{proposition}
\label{lemmaRelaxed}
Let \(N \in \N\) and \(p \in [2, +\infty)\).
If the function \(V : \R^N \to \R\) satisfies the assumptions \eqref{V1} and \eqref{V2} of \cref{theoremMainGroundstate}, then for every \(\lambda \in [0, +\infty)\) the supremum \(a_\lambda\) is achieved.
\end{proposition}

The proof is inspired by the proof of the corresponding property for the Choquard equation with a Riesz potential
\cite{mv13}*{Proposition 2.2} with the additional difficulty that we have to take care of the convergence of the term involving $V^-$, for which we have less information (see \cref{lemmaBrezisLieb}).

\begin{proof}%
[Proof of \cref{lemmaRelaxed}]
\resetconstant
By the assumption \eqref{V3} we have \(a_\lambda > 0\).
Let $(w_n)_{n \in \N}$ be a sequence in \(H^1 \bigl(\R^N\bigr)\) satisfying
\begin{align}%
\label{wn}
      &\int_{\R^N}
        \bigl(\abs{\nabla w_n}^2+\abs{w_n}^2\bigr)
    =   
      1,&
  &\text{ and }&
    &\lim_{n \to \infty}       
      \int_{\R^N}
        \bigl(V_\lambda *\abs{w_n}^p\bigr)\abs{w_n}^p 
  = 
    a_\lambda.
\end{align}
Since the sequence $(w_n)_{n \in \N}$ is bounded in the space $H^1\bigl(\R^N\bigr)$, it converges weakly, up to a subsequence, to some function $w\in H^1\bigl(\R^N\bigr)$.
We first show that we may assume that $w\not\equiv 0$ on \(\R^N\). 

By an inequality of P.-L. Lions \cite{lio}*{Lemma I.2} (see also \citelist{\cite{Willem1996}*{lemma 1.21}\cite{mv13}*{lemma 2.3}\cite{VanSchaftingen2014}*{(2.4)}}), we have
\[
\begin{split}
\int_{\R^N}\abs{w_n}^\frac{2pq}{2q-1}
&\le \Cl{Clions} \int_{\R^N}\bigl(\abs{\nabla w_n}^2+\abs{w_n}^2\bigr) \;\left(\sup_{x\in\R^N}\int_{B_1(x)}\abs{w_n}^\frac{2pq}{2q-1}\right)^{1-\frac{1}{p} (2 - \frac{1}{q})}\\
&=\Cr{Clions} \left(\sup_{x\in\R^N}\int_{B_1(x)}\abs{w_n}^\frac{2pq}{2q-1}\right)^{1-\frac{1}{p} (2 - \frac{1}{q})}
.
\end{split}
\]
From this, we get, in view of Young's convolution inequality (\cref{young}),
\begin{equation*}
\begin{split}
\left(\sup_{x\in\R^N}\int_{B_1(x)}\abs{w_n}^\frac{2pq}{2q-1}\right)^{1-\frac{2q-1}{pq}}
%
%
\ge \frac{\displaystyle \biggl( \int_{\R^N}
\bigl(V^+ \ast \abs{w_n}^p\bigr) \abs{w_n}^p\biggr)^\frac{q}{2q-1}}{\Cr{Clions} \displaystyle \Bigl(\int_{\R^N} \abs{V^+}^q\Bigr)^\frac{1}{2 q - 1} }.
\end{split}
\end{equation*}
Hence we deduce that 
\[
 \C \limsup_{n \to \infty}
\left(\sup_{x\in\R^N}\int_{B_1(x)}\abs{w_n}^\frac{2pq}{2q-1}\right)^{(2 - \frac{1}{q})(1-\frac{2q-1}{pq})}
\ge \lim_{n \to \infty} \int_{\R^N}
\bigl(V_\lambda \ast \abs{w_n}^p\bigr) \abs{w_n}^p
= a_\lambda > 0.
\]
Since \(\frac{1}{p} (1 - \frac{1}{2 q}) < \frac{1}{p} \le \frac{1}{2}\), this implies that there exists a sequence of points \((x_n)_{n \in \N}\) in \(\R^N\) such that 
\begin{equation}
\label{eqNonvanishing}
  \liminf_{n\to+\infty}\int_{B_1(x_n)}\abs{w_n}^\frac{2pq}{2q-1}>0.
\end{equation}
By replacing for each \(n\in \N\) the function \(w_n\) by its translation \(w_n(\cdot+x_n)\) we can assume that \eqref{eqNonvanishing} holds with \(x_n = 0\), and thus by the Rellich--Kondrashov compact embedding theorem that
\[
 \int_{B_1} \abs{w}^2 = \lim_{n \to \infty} \int_{B_1} \abs{w_n}^2 > 0,
\]
so that \(w \not \equiv 0\) on \(\R^N\).

By its weak convergence, the sequence $(w_n)_{n \in \N}$ also satisfie
\begin{equation*}
\int_{\R^N}\bigl(\abs{\nabla w}^2+\abs{w}^2\bigr)=\lim_{n \to \infty} \int_{\R^N}\bigl(\abs{\nabla w_n}^2+\abs{w_n}^2\bigr)-\int_{\R^N}\bigl(\abs{\nabla (w_n - w)}^2+\abs{w_n - w}^2\bigr); 
\end{equation*}
Since the sequence \((w_n)_{n \in \N}\) is bounded in \(H^1 (\R^N)\), we have, by the Sobolev embedding theorem and by \cref{lemmaBrezisLieb}
\begin{equation*}
\int_{\R^N}(V*\abs{w}^p)\abs{w}^p\ge\limsup_{n \to \infty} \int_{\R^N}\bigl(V*\abs{w_n}^p\bigr)\abs{w_n}^p-\int_{\R^N}(V*\abs{w_n - w}^p)\abs{w_n - w}^p.
\end{equation*}
By homogeneity, any function $v\in H^1\bigl(\R^N\bigr)\setminus\{0\}$ verifies
$$
\frac{\displaystyle \int_{\R^N}(V_\lambda*\abs{v}^p)\abs{v}^p}{\left(\displaystyle\int_{\R^N}\bigl(\abs{\nabla v}^2+\abs{v}^2\bigr) \right)^p}\le a_\lambda.
$$
Therefore,
\begin{equation}
\begin{split}
\label{nonvan1}
\frac{\int_{\R^N}(V_\lambda *\abs{w}^p)\abs{w}^p}{\left(\int_{\R^N}\bigl(\abs{\nabla w}^2+\abs{w}^2\bigr)\right)^p}
\ge&\limsup_{n \to \infty} \frac{\int_{\R^N}\bigl(V_\lambda *\abs{w_n}^p\bigr)\abs{w_n}^p-\int_{\R^N}(V_\lambda *\abs{w_n - w}^p)\abs{w_n - w}^p}{\left(\int_{\R^N}\bigl(\abs{\nabla w}^2+\abs{w}^2\bigr)\right)^p}\\
=&\limsup_{n \to \infty} \frac{\int_{\R^N}\bigl(V_\lambda *\abs{w_n}^p\bigr)\abs{w_n}^p}{\left(\int_{\R^N}\left(\abs{\nabla w_n}^2+\abs{w_n}^2\right)\right)^p}\left(\frac{\int_{\R^N}\left(\abs{\nabla w_n}^2+\abs{w_n}^2\right)}{\int_{\R^N}(\abs{\nabla w}^2+\abs{w}^2)}\right)^p\\
& \qquad \qquad-\frac{\int_{\R^N}(V_\lambda *\abs{w_n - w}^p)\abs{w_n - w}^p}{\left(\int_{\R^N}\left(\abs{\nabla (w_n - w)}^2+\abs{w_n - w}^2\right)\right)^p}\\
&\qquad \qquad \quad \qquad \times\left(\frac{\int_{\R^N}(\abs{\nabla (w_n - w)}^2+\abs{w_n - w}^2)}{\int_{\R^N}(\abs{\nabla w}^2+\abs{w}^2)}\right)^p\\
\end{split}
\end{equation}
By the optimizing character of the sequence \((w_n)_{n \in \N}\) and by definition of \(a_\lambda\), we have
\begin{equation}
\begin{split}
\label{nonvan2}
\frac{\int_{\R^N}(V_\lambda *\abs{w}^p)\abs{w}^p}{\left(\int_{\R^N}\bigl(\abs{\nabla w}^2+\abs{w}^2\bigr)\right)^p}
\ge &a_\lambda \limsup_{n \to \infty} \Biggl(\left(\frac{\int_{\R^N}\left(\abs{\nabla w_n}^2+\abs{w_n}^2\right)}{\int_{\R^N}(\abs{\nabla w}^2+\abs{w}^2)}\right)^p\\
&\quad\qquad \qquad-\left(\frac{\int_{\R^N}(\abs{\nabla (w_n - w)}^2+\abs{w_n - w}^2)}{\int_{\R^N}(\abs{\nabla w}^2+\abs{w}^2)}\right)^p\Biggr)\\
=&a_\lambda \limsup_{n \to \infty} \Biggl(\left(\frac{\int_{\R^N}(\abs{\nabla (w_n - w)}^2+\abs{w_n - w}^2)}{\int_{\R^N}(\abs{\nabla w}^2+\abs{w}^2)}+1\right)^p\\
&\qquad \qquad \qquad -\left(\frac{\int_{\R^N}(\abs{\nabla (w_n - w)}^2+\abs{w_n - w}^2)}{\int_{\R^N}(\abs{\nabla w}^2+\abs{w}^2)}\right)^p\Biggr)\\
\ge& a_\lambda.
\end{split}
\end{equation}
It follows then that the function $u_\lambda = w/\sqrt{\int_{\R^N}(\abs{\nabla w}^2+\abs{w}^2)}$ satisfies the required properties.
\end{proof}

\begin{remark}
In the proof of \cref{lemmaRelaxed}, the last chain of inequalities \eqref{nonvan1}--\eqref{nonvan2} must actually be a chain of equalities, and in particular
$$
\limsup_{n\to+\infty}\int_{\R^N}\bigl(\abs{\nabla w_n}^2+\abs{w_n}^2\bigr) = \int_{\R^N}\bigl(\abs{\nabla w}^2+\abs{w}^2\bigr),
$$
so that the sequence \((w_n)_{n \in \N}\) converges strongly to \(w\) in $H^1\bigl(\R^N\bigr)$ and the function $w$ itself attains $a_\lambda$.
\end{remark}

%

\section{Groundstates for the Choquard equation}
\label{sectionGroundState}
This section is devoted to the proof of \cref{theoremMainGroundstate} about the existence of groundstates for the Choquard equation \eqref{equationChoquard}.

The proof of \cref{theoremMainGroundstate} will use the fact that coarsely continuous functions are large-scale Lipschitz-continuous.

\begin{lemma}\label{sr}
If \(f : \R^N \to \R\), then for every \(x, y \in \R^N\),
\[
 \abs{f (x) - f (y)} 
 \le (\abs{x - y} + 1) \,\sup\, \bigl\{ \abs{f (z) - f (w)} \st z, w \in \R^N \text{ and } \abs{z - w} \le 1\bigr\}.
\]
\end{lemma}

\begin{proof}
We take points $x_0, \dotsc,x_{\ell}$, with \(x_0 = x\), \(x_\ell = y\), \(\ell \le |x-y| + 1\), \( \abs{x_i-x_{i+1}} \le 1$ for $i \in \{0, \dotsc, \ell - 1\}$, and we estimate, by the triangle inequality,
\begin{equation}
\begin{split}
    \abs{f (x) - f (y)} 
  \le &
    \sum_{i=0}^{\ell-1}\abs{f(x_i)-f(x_{i+1})}\\
  \le 
    &\Bigl(\sum_{i=0}^{\ell-1}|x_i-x_{i+1}|\Bigr) 
      \,\sup\, \bigl\{ \abs{f (z) - f (w)} \st z, w \in \R^N \text{ and } \abs{z - w} \le 1\bigr\}.\\
  \le &
    (\abs{x - y}+1) 
    \,\sup\, \bigl\{ \abs{f (z) - f (w)} \st z, w \in \R^N \text{ and } \abs{z - w} \le 1\bigr\}. \qedhere
\end{split}
\end{equation}
\end{proof}

\Cref{definitionWeak} of weak solutions does not allow a priori to test the equation against against a solution. We show that the resulting formula still holds however. 

\begin{lemma}[Testing the equation against a solution]
\label{lemmaTest}
If \(p \ge 2\), if \(V \in L^q (\R^N)\) with \( \frac{1}{p}(1 - \frac{1}{2q}) \ge \frac{1}{2} - \frac{1}{N}\) and if \(u \in H^1 (\R^N)\) is a weak solution to the Choquard equation \eqref{equationChoquard}, then
\(\int_{\R^N} \bigl(\abs{V} \ast \abs{u}^p\bigr)\, \abs{u}^p < + \infty\) and 
\[
 \int_{\R^N} (V \ast \abs{u}^p) \abs{u}^p = 
 \int_{\R^N} \bigl( \abs{\nabla u}^2 + \abs{u}^2 \bigr).
\]
\end{lemma}
\begin{proof}
We consider a function \(\eta \in C^2_c (\R^N)\) such that \(\eta = 1\) on \(B_1\), \(0 \le \eta \le 1\) on \(\R^N\) and for each \(x \in \R^N\) the function \(t \in [0, +\infty) \mapsto \eta (t x)\) is nonincreasing,
and we define for each \(R > 0\) the function \(\eta_R : \R^N \to \R\) for each \(x \in \R^N\) by \(\eta_R (x) = \eta (x/R)\). 
Since \(u \in H^1 (\R^N)\) and \(\eta_R\) has compact support, the function
\(\eta_R u\) is an admissible test function of the weak formulation of the Choquard equation (\cref{definitionWeak}). Hence, we have 
\[
 \int_{\R^N} \bigl( \eta_R \abs{\nabla u}^2
 + \eta_R \abs{u}^2 + u \nabla u \cdot \nabla \eta_R \bigr)
 = \int_{\R^N} (V \ast \abs{u}^p) \abs{u}^p \eta_R.
\]
Therefore, since \(\eta_R \le 1\), since \(\abs{\nabla \eta_R} \le \norm{\nabla \eta}_{L^\infty}/R\) and since, by combining our assumption with the classical Sobolev embedding, we have \(u \in L^{\frac{2 p q}{2 q - 1}}(\R^N)\) so that by Young's convolution inequality (\cref{young}) we have \(\int_{\R^N}  (V^+ \ast \abs{u}^p) \abs{u}^p < + \infty\),  we obtain by Lebesgue's dominated convergence theorem that 
\begin{multline*}
\lim_{R \to \infty}
 \int_{\R^N} \bigl( \eta_R \abs{\nabla u}^2
 + \eta_R \abs{u}^2 + u \nabla u \cdot \nabla \eta_R \bigr)
 -  \int_{\R^N} (V^+ \ast \abs{u}^p) \abs{u}^p \eta_R\\
 = \int_{\R^N} \bigl( \abs{\nabla u}^2 + \abs{u}^2 \bigr)-  \int_{\R^N} (V^+ \ast \abs{u}^p) \abs{u}^p .
\end{multline*}
Hence by Lebesgue's monotone convergence theorem 
\[
\begin{split}
  \int_{\R^N} (V^- \ast \abs{u}^p) \abs{u}^p
  &= \lim_{R \to \infty} 
  \int_{\R^N} (V^- \ast \abs{u}^p) \abs{u}^p \eta_R\\
  &= \int_{\R^N} (V^+ \ast \abs{u}^p) \abs{u}^p  - \int_{\R^N} \bigl( \abs{\nabla u}^2 + \abs{u}^2 \bigr) < + \infty.
\end{split}
\]
The conclusions then follow.
\end{proof}

\begin{proof}[Proof of \cref{theoremMainGroundstate}]
\resetconstant
\resetclaim
By \cref{lemmaRelaxed}, for every \(\lambda > 0\), there exists a function \(w_\lambda \in H^1 (\R^N)\) such that 
\begin{align*}
 \int_{\R^N}\bigl( \abs{\nabla w_\lambda}^2 + \abs{w_\lambda}^2 \bigr)&= 1
 &\text{ and }&
 & \int_{\R^N} \bigl(V_\lambda \ast \abs{w_\lambda}^p) \abs{w_\lambda}^p & = 
 a_\lambda,
\end{align*}
where \(a_\lambda\) was defined in \eqref{eqDef_a_lambda}.

\begin{claim}
\label{claim_a_lambda}
For every \(\lambda \in [0, +\infty)\), we have 
\[
 0 < \lim_{\mu \to \infty} a_\mu \le a_\lambda \le a_0< + \infty.
\]
\end{claim}
\begin{proofclaim}
We observe that if \(\lambda_1 \le \lambda_2\), then by the definition of the relaxed potential \(V_\lambda\) in \eqref{eqDef_relaxed_potential}, we have \(V_{\lambda_1} \ge V_{\lambda_2}\) and thus by the definition of \(a_\lambda\) in \eqref{eqDef_a_lambda}, we have 
\(a_{\lambda_1} \ge a_{\lambda_2}\). In particular, we have \(a_\lambda \le a_0\).

Moreover by definition of \(V_\lambda\) in \eqref{eqDef_relaxed_potential} again, we have for every \(\lambda >0\) and \(\varphi \in C^1_c (\R^N)\).
\[
 a_\lambda \ge \frac{\displaystyle \int_{\R^N} \bigl(V_\lambda \ast \abs{\varphi}^p\bigr) \abs{\varphi}^p}
 {\biggl( \displaystyle \int_{\R^N} \bigl( \abs{\nabla \varphi}^2 + \abs{\varphi}^2\bigr) \biggr)^p}
 \ge \frac{\displaystyle \int_{\R^N} \bigl(V_\lambda \ast \abs{\varphi}^p\bigr) \abs{\varphi}^p}
 {\biggl( \displaystyle \int_{\R^N}\bigl( \abs{\nabla \varphi}^2 + \abs{\varphi}^2 \bigr) \biggr)^p} > 0.
\]
By the assumptions \eqref{V3} and \eqref{V4}, the right-hand side can be chosen to be positive and independent of \(\lambda\), so that the conclusion follows. 
\end{proofclaim}

\begin{claim}
\label{claimCompactness}
There exists a sequence \((\lambda_n)_{n \in \N}\) in \([0, + \infty)\) such that \(\lim_{n \to \infty} \lambda_n = + \infty\), a sequence \((x_n)_{n \in \N}\) in \(\R^N\) and 
a function \(w \in H^1 (\R^N)\) such that if 
\(
 w_n = w_{\lambda_n} (\cdot + x_n)
\),
then 
\begin{enumerate}[(1)]
 \item  \label{claimNontrivial} \(w \not \equiv 0\) in \(\R^N\),
 \item \label{claimWeak} the sequence \((w_n)_{n \in \N}\) converges weakly to \(w\) in \(H^1 (\R^N)\),
 \item \label{claim_a_e} the sequence \((w_n)_{n \in \N}\) converges almost everywhere to \(w\) in \(\R^N\),
 \item \label{claim_glob} the sequence \((w_n)_{n \in \N}\) converges strongly to \(w\) in \(L^p (\R^N)\),
 \item \label{claim_ConvergenceEquation} there exists a constant \(\mu \in [0, + \infty)\) such that for every \(R >0\),
 the sequence \((V_{\lambda_n}^- \ast \abs{w_n}^p)_{n \in \N}\) converges to 
 \((V^- \ast \abs{w}^p) + \mu\) in \(L^\infty (B_R)\).
\end{enumerate}
\end{claim}

\begin{proofclaim}

By Young's convolution inequality (\cref{young}), we have
\begin{equation}
\label{eq_Lions2_Young}
\int_{\R^N}\abs{w_\lambda}^\frac{2pq}{2q-1}
\ge \frac{\displaystyle \biggl( \int_{\R^N}
\bigl(V^+ \ast \abs{w_\lambda}^p\bigr) \abs{w_\lambda}^p\biggr)^\frac{q}{2q-1}}{\displaystyle \Bigl(\int_{\R^N} \abs{V^+}^q\Bigr)^\frac{1}{2 q - 1} }
\ge C a_\lambda^\frac{q}{2 q - 1}. 
\end{equation}
By an inequality of P.-L. Lions \cite{lio}*{Lemma I.2} (see also \citelist{\cite{Willem1996}*{lemma 1.21}\cite{mv13}*{lemma 2.3}\cite{VanSchaftingen2014}*{(2.4)}}), we have on the one hand 
\begin{equation}
\label{eq_Lions2}
\begin{split}
\int_{\R^N}\abs{w_\lambda}^\frac{2pq}{2q-1}
&\le \Cl{Clions2} \int_{\R^N}\bigl(\abs{\nabla w_\lambda}^2+\abs{w_\lambda}^2\bigr) \;\left(\sup_{x\in\R^N}\int_{B_1(x)}\abs{w_\lambda}^\frac{2pq}{2q-1}\right)^{1-\frac{1}{p} (2 - \frac{1}{q})}\\
&=\Cr{Clions2} \left(\sup_{x\in\R^N}\int_{B_1(x)}\abs{w_\lambda}^\frac{2pq}{2q-1}\right)^{1-\frac{1}{p} (2 - \frac{1}{q})}
.
\end{split}
\end{equation}
By combining \eqref{eq_Lions2_Young} and \eqref{eq_Lions2}, we deduce that there exists a family \((x_{\lambda})_{\lambda \ge 0}\) of points in \(\R^N\) such that 
\begin{equation}
\label{eq_Lions2_conclusion}
  \Cr{Clions2} \liminf_{\lambda \to \infty} \left(\int_{B_1(x_\lambda)}\abs{w_\lambda}^\frac{2pq}{2q-1}\right)^{1-\frac{1}{p} (2 - \frac{1}{q})}
  \ge \liminf_{\lambda \to \infty} a_\lambda^\frac{q}{2 q - 1} > 0,
\end{equation}
by \cref{claim_a_lambda}.

There exists thus a sequence \((\lambda_n)_{n \in \N}\) in \([0, + \infty)\) such that \(\lim_{n \to \infty} \lambda_n= + \infty\) and the sequence \((w_{\lambda_n} (\cdot + x_{\lambda_n}))_{n \in \N}\) converges weakly to some function \(w \in H^1 (\R^N)\). This proves \eqref{claimWeak}.
By Rellich's compactness theorem, it follows that the sequence \((w_n)_{n \in \N}\) converges to \(w\) strongly in \(L^r_{\mathrm{loc}} (\R^N)\) whenever \(\frac{1}{r} > \frac{1}{2} - \frac{1}{N}\).
From the convergence in \(L^{2 pq/(2 q - 1)}_{\mathrm{loc}} (\R^N)\), we deduce in view of \eqref{eq_Lions2_conclusion} that 
\begin{equation}
\label{eq_dpxsar}
 \int_{B_1} \abs{w}^\frac{2 pq}{2 q - 1}
 = \lim_{n \to \infty} \int_{B_1} \abs{w_n}^\frac{2 pq}{2 q - 1}
 = \lim_{n \to \infty} \int_{B_1(x_{\lambda_n})}\abs{w_{\lambda_n}}^\frac{2pq}{2q-1} > 0,
\end{equation}
and \eqref{claimNontrivial} follows.
The assertion \eqref{claim_a_e} then follows up to extraction of a subsequence from the strong convergence in \(L^2_{\mathrm{loc}} (\R^N)\). 

In order to prove \eqref{claim_glob}, 
we first note that 
\begin{equation}%
\label{eq_idspsxt}
\begin{split}
 \limsup_{n \to \infty} \int_{\R^N} \abs{w_n - w}^p
 &\le \limsup_{R \to \infty} 2^{p - 1} \limsup_{n \to \infty} \int_{\R^N \setminus B_R} \abs{w_n}^p 
 + \limsup_{R \to \infty} 2^{p - 1} \int_{\R^N \setminus B_R} \abs{w}^p\\
 &\qquad + \limsup_{R \to \infty} \limsup_{n \to \infty} \int_{B_R} \abs{w_n - w}^p.
\end{split}
\end{equation}
By Rellich's compactness theorem, we have for every \(R \in (0, +\infty)\),
\begin{equation}
\label{eq_idspsxt_1}
 \limsup_{n \to \infty} \int_{B_R} \abs{w_n - w}^p = 0.
\end{equation}
Next, we observe that, by Fatou's lemma,
\[
  \int_{\R^N} \abs{w}^p \le \liminf_{n \to \infty} \int_{\R^N} \abs{w_n}^p < + \infty, 
\]
so that by Lebesgue's dominated convergence theorem, it follows that 
\begin{equation}
 \label{eq_idspsxt_2}
  \limsup_{R \to \infty}  \int_{\R^N \setminus B_R} \abs{w}^p = 0.
\end{equation}
In order to conclude, we observe that for each \(n \in \N\), by definition of \(a_{\lambda_n}\) we have
\[
\begin{split}
\int_{\R^N} \bigl(V^+ \ast \abs{w_n}^p\bigr)\abs{w_n}^p
& \ge 
  \int_{\R^N} \bigl(V^+ \ast \abs{w_n}^p\bigr)
 \abs{w_n}^p - a_{\lambda_n}\\
 & = \int_{\R^N} \bigl(V_{\lambda_n}^- \ast \abs{w_n}^p\bigr)
 \abs{w_n}^p\\
 & \ge \int_{\R^N \setminus B_R} \biggl(\int_{B_1} V_{\lambda_n}^- (x - y) \abs{w_n (y)}^p \dif y\biggr) \abs{w_n (y)}^p \dif y \\
 & \ge \Bigl(\inf_{\R^N \setminus B_{R + 1}} V_{\lambda_n}^- \Bigr) \biggl(\int_{\R^N \setminus B_R} \abs{w_n}^p\biggr)
 \biggl(\int_{B_1} \abs{w_n}^p \biggr)
\end{split}
\]
The left-hand side is bounded in view of \eqref{V1} and by Young's convolution inequality (\cref{young}). Therefore, we have 
\[
 \limsup_{n \to \infty} \int_{\R^N \setminus B_R} \abs{w_n}^p 
 \le \frac{\C}{\bigl(\inf_{\R^N \setminus B_{R + 1}} V^-\bigr) \displaystyle\int_{B_1} \abs{w_n}^p}.
\]
By the assumption \eqref{V4} and by \eqref{eq_dpxsar}, we deduce that 
\begin{equation}
\label{eq_idspsxt_3}
 \lim_{R \to \infty} \limsup_{n \to \infty} \int_{\R^N \setminus B_R} \abs{w_n}^p 
 = 0.
\end{equation}
Therefore, in view of \eqref{eq_idspsxt}, \eqref{eq_idspsxt_1}, \eqref{eq_idspsxt_2} and \eqref{eq_idspsxt_3}, we have
\begin{equation}
 \limsup_{n \to \infty} \int_{\R^N} \abs{w_n - w}^p
 = 0.
\end{equation}
The assertion \eqref{claim_glob} follows then.

For the assertion \eqref{claim_ConvergenceEquation}, we first observe that for each \(n \in \N\), 
\[
 \inf_{B_1} \,\bigl(V_{\lambda_n}^- \ast \abs{w_n}^p\bigr)
 \le \frac{\displaystyle \int_{\R^N} (V_{\lambda_n}^- \ast \abs{w_n}^p)|w_n|^p}
 {\displaystyle \int_{B_1} \abs{w_n}^p} \le \C,
\]
in view of \eqref{eq_dpxsar}.
Therefore, there exists a sequence of points \((x_n)_{n \in \N}\) in \(B_1\) such that 
\[
 \limsup_{n \to \infty} \,\bigl(V_{\lambda_n}^- \ast \abs{w_n}^p\bigr) (x_n) < + \infty.
\]
By assumption \eqref{V2} and by \cref{sr}, we have in view of the boundedness of the sequence \((w_n)_{n \in \N}\) in \(L^p (\R^N)\),
\[
\begin{split}
    0 
  \le 
    \bigl(V_{\lambda_n}^- \ast \abs{w_n}^p\bigr) (0)
  \le 
      \bigl(V_{\lambda_n}^- \ast \abs{w_n}^p\bigr) (x_n)
    + 
      \int_{\R^N} 
        \bigabs{V (-y) - V (x_n - y)} \abs{w_n (y)}^p 
      \dif y
 \le 
    \C.
\end{split}
\]
This implies by Fatou's lemma that 
\[
    0 
  \le 
    \bigl(V \ast \abs{w}^p\bigr) (0)
  \le 
    \liminf_{n \to \infty} 
      \bigl(V_{\lambda_n}^- \ast \abs{w_n}^p\bigr) (0) 
  < 
    + \infty.
\]
The sequence \(((V_{\lambda_n}^- \ast \abs{w_n}^p) (0))_{n \in \N}\) is bounded and, up to the extraction of a subsequence, we can assume that the sequence of real numbers $\bigl( (V_{\lambda_n}^- \ast \abs{w_n}^p) (0) - (V^- \ast \abs{w}^p) (0)\bigr)_{n \in \N}$ converges to some \(\mu \in [0, + \infty)\). 
Moreover by Lebesgue's monotone convergence theorem, we have,
\begin{equation}
\label{eq_vdid_idt}
 \bigl(V \ast \abs{w}^p\bigr) (0)
= \lim_{n \to \infty }\bigl(V_{\lambda_n}^- \ast \abs{w}^p\bigr) (0).
\end{equation}

Next, we observe that for every \(x \in \R^N\) and \(n \in \N\), by assumption \eqref{V2} and by \cref{sr}
\begin{multline*}
\bigabs{(V_{\lambda_n}^-*|w_n|^p)(x)-(V_{\lambda_n}^-*\abs{w}^p)(x)-(V_{\lambda_n}^-*|w_n|^p)(0)+(V_{\lambda_n}^-*\abs{w}^p)(0)}\\
 \le\int_{\R^N}\bigabs{ V_{\lambda_n}^- (x-y) - V_{\lambda_n}^-(-y)} \;
\bigabs{ \abs{w_n(y)}^p-\abs{w(y)}^p}\dif y\\
\le \C (\abs{x} + 1) \int_{\R^N} \bigabs{|w_n|^p-\abs{w}^p},
\end{multline*}
and thus by \eqref{claim_glob} and \eqref{eq_vdid_idt}, we have 
\[
  \lim_{n \to \infty} \sup_{x \in \R^N} \frac{\bigabs{(V_{\lambda_n}^-*|w_n|^p)(x) - (V^- \ast \abs{w}^p)(x) - \mu}}{1 + \abs{x}} = 0,
\]
and the assertion \eqref{claim_ConvergenceEquation} follows.
\end{proofclaim}

\begin{claim}
\label{claimConvergenceVariationnal}
We have 
\begin{align*}
 \int_{\R^N} \bigl(\abs{\nabla w}^2 + \abs{w}^2\bigr) & = 1 &
 &\text{ and }&
  \int_{\R^N}\bigl(V*\abs{w}^p\bigr)\,\abs{w}^p=a = \lim_{\lambda \to \infty} a_\lambda
\end{align*}
and the sequence \((w_n)_{n \in \N}\) converges strongly to \(w\) in \(H^1 (\R^N)\).
\end{claim}

\begin{proofclaim}
We first observe that, by the definitions of \(a\) in \eqref{a} and of \(a_\lambda\) in \eqref{eqDef_a_lambda}, in view of the fact that \(V_\lambda \ge V\) by definition in \eqref{eqDef_relaxed_potential}, we have 
\(a_\lambda \ge a\).

By \eqref{claim_glob} in \cref{claimCompactness}, the Sobolev embedding and the H\"older inequality, the sequence \((w_n)_{n \in \N}\) converges strongly to \(w\) in \(L^{\frac{2 pq}{2 q - 1}} (\R^N)\) and thus by Young's convolution inequality (\cref{young})
\[
  \lim_{n \to \infty}
  \int_{\R^N}\bigl(V^+ * \abs{w_n}^p\bigr) \,\abs{w_n}^p
  = \int_{\R^N}\bigl(V^+*\abs{w}^p\bigr)\,\abs{w}^p.
\]
In view of \eqref{claim_a_e} in \cref{claimCompactness}, the sequence \((\abs{w_n}^p)_{n \in \N}\) converges to \(\abs{w}^p\) almost everywhere in \(\R^N\). Since \(V_\lambda^- \to V^-\) everywhere in \(\R^N\) as \(\lambda \to \infty\), we deduce by an application of Fatou's lemma on \(\R^N \times \R^N\) that 
\[
 \liminf_{n \to \infty} \int_{\R^N}\bigl(V_\lambda^- * \abs{w_n}^p\bigr)\, \abs{w_n}^p 
 \ge \int_{\R^N} \bigl(V^- * \abs{w}^p\bigr)\, \abs{w}^p.
\]
Hence we have 
\begin{equation}
\label{eqConvergence_one}
  \int_{\R^N} \bigl(V * \abs{w}^p\bigr) \abs{w}^p
  \ge \limsup_{n \to \infty} \int_{\R^N}\bigl(V * \abs{w_n}^p\bigr) \abs{w_n}^p 
  = \lim_{n \to \infty} a_{\lambda_n} = a.
\end{equation}
On the other hand we have, by definition of \(a\) in \eqref{a},
\begin{equation}
\label{eqConvergence_two}
 \int_{\R^N} \bigl(V * \abs{w}^p\bigr) \abs{w}^p
 \le a \biggl(\int_{\R^N}\bigl( \abs{\nabla w}^2 + \abs{w}^2\bigr)\biggr)^p
 \le a \liminf_{n \to \infty} \int_{\R^N}\bigl( \abs{\nabla w_n}^2 + \abs{w_n}^2\bigr)
 = a.
\end{equation}
The claim then follows from \eqref{eqConvergence_one} and \eqref{eqConvergence_two}.
\end{proofclaim}

\begin{claim}
\label{claimEquation} We have
\(V \ast \abs{w}^p \in L^s_{\mathrm{loc}} (\R^N)\) for every \(s \in [1, + \infty)\) such that \(\frac{1}{s} \ge p (\frac{1}{2} - \frac{1}{N}) + \frac{1}{q} - 1\) and for every \(\varphi \in C^1_c (\R^N)\),
\[
  a \int_{\R^N} \nabla w \cdot \nabla \varphi + w\, \varphi
  = \int_{\R^N} \bigl(V \ast \abs{w}^p\bigr) \abs{w}^{p - 2} w\, \varphi.
\]
\end{claim}
\begin{proofclaim}
First, for every \(\varphi \in C^1_c (\R^N)\),
\begin{equation}
\label{eqWeakRelaxedLagrange}
 a_{\lambda_n} \int_{\R^N} \nabla w_n \cdot \nabla \varphi + w_n \varphi
  = \int_{\R^N} \bigl(V_{\lambda_n} \ast \abs{w_n}^p\bigr) \abs{w_n}^{p - 2} w_n \,\varphi
\end{equation}
By \cref{claimCompactness} and Sobolev's embedding theorem, the sequence \((w_n)_{n \in \N}\) converges strongly to \(w\) in \(L^r (\R^N)\) if \(\frac{1}{r} \ge \frac{1}{2} - \frac{1}{N}\). By Young's convolution inequality (\cref{young}), the sequence \((V^+ \ast \abs{w_n}^p)_{n \in \N}\) converges strongly to \((V^+ \ast \abs{w}^p)_{n \in \N}\) in \(L^s_{\mathrm{loc}} (\R^N)\) whenever \(\frac{1}{s} \ge p (\frac{1}{2} - \frac{1}{N}) + \frac{1}{q} - 1\).
On the other hand, since for each \(n \in \N\), the function \(V_{\lambda_n}^-\) is bounded and \(w_n \in L^p (\R^N)\), the function  \(V_{\lambda_n}^- \ast \abs{w_n}^p\) is also bounded and we deduce from \cref{claimCompactness} that the sequence \((V_{\lambda_n}^- \ast \abs{w_n}^p)_{n \in \N}\) converges to \(\mu + V^- \ast \abs{w}^p\) in \(L^\infty_{\mathrm{loc}} (\R^N)\).
Therefore, by letting \(n \to \infty\) in \eqref{eqWeakRelaxedLagrange}, we obtain
\begin{equation}
\label{eqWeakChoquardLagrange}
 a \int_{\R^N} \nabla w \cdot \nabla \varphi + w \,\varphi
  = \int_{\R^N} \bigl(V \ast \abs{w}^p\bigr) \abs{w}^{p - 2} w\, \varphi
  + \mu \int_{\R^N} \abs{w}^{p - 2} w \,\varphi.
\end{equation}
Since \(V \ast \abs{w}^p \in L^s (\R^N)\) whenever \(\frac{1}{s} \ge p (\frac{1}{2} - \frac{1}{N}) + \frac{1}{q} - 1\), the identity \eqref{eqWeakChoquardLagrange} still holds when \(\varphi \in H^1 (\R^N)\) has compact support by the density of smooth functions in the Sobolev spaces.
We conclude the proof of the claim by proving that \(\mu = 0\). 

By a variant of \cref{lemmaTest} and by \cref{claimConvergenceVariationnal}, we have
\[
\begin{split}
 \int_{\R^N} \bigl(V \ast \abs{w}^p\bigr)\, \abs{w}^p  &= a \int_{\R^N} \bigl(\abs{\nabla w}^2  + \abs{w}^2\bigr)\\
 &=\int_{\R^N} \bigl(V \ast \abs{w}^p\bigr)\, \abs{w}^p - \mu \int_{\R^N} \abs{w}^p,
\end{split}
\]
which can only hold if \(\mu = 0\).
\end{proofclaim}

We conclude now the proof of \cref{theoremMainGroundstate}, by setting \(u = a^\frac{-1}{2p - 2} w\). By \cref{claimEquation}, the function $u$ satisfies the Choquard equation \eqref{equationChoquard}.
Let us show that $u$ is a groundstate.
By a direct computations, its energy value is
\[
\mathcal{I} (u)
=\frac{a_\lambda^{-\frac{1}{p-1}}}2\int_{\R^N}\bigl(\abs{\nabla w}^2+\abs{w}^2\bigr)-\frac{a_\lambda^{-\frac{p}{p-1}}}{2p}\int_{\R^N}(V_\lambda *\abs{w}^p)\abs{w}^p
=\Bigl(\frac{1}{2}-\frac{1}{2p}\Bigr) a_\lambda^{-\frac{1}{p-1}}.
\]
If \(v \in H^1 (\R^N)\) is a weak solution to the Choquard equation \eqref{equationChoquard}, then we have by \cref{lemmaTest}
\[
    a_\lambda 
  \ge 
    \frac
      {\displaystyle 
        \int_{\R^N} 
          \bigl(V \ast \abs{v}^p\bigr)\, \abs{v}^p}
      {\biggl(\displaystyle
        \int_{\R^N}\bigl( \abs{\nabla v}^2 + \abs{v}^2 \bigr)\biggr)^p}
  =
    \bigl(\tfrac{2p}{p - 1} \,\mathcal{I} (v)\bigr)^{p - 1},
\]
and the conclusion follows.
\end{proof}

\end{document}